\documentclass[11pt]{amsart}
\usepackage{amsmath,amssymb,amscd,amsfonts,verbatim}
\usepackage{paralist}
\usepackage[mathscr]{eucal}
\usepackage{hyperref}
\usepackage{pdfsync}

\newtheorem{thm}{Theorem}[section]
\newtheorem{lem}[thm]{Lemma}

\newtheorem{prop}[thm]{Proposition}
\newtheorem{cor}[thm]{Corollary}

\newtheorem{assu-nota}[thm]{Assumption--Notation}
\theoremstyle{remark}

\newcommand{\Q}{\mathbb Q}

\newcommand{\pp}{\mathbb P}

\newcommand{\OO}{\mathcal O}

\DeclareMathOperator{\Pic}{Pic}

\DeclareMathOperator{\Alb}{Alb}
\DeclareMathOperator{\albdim}{albdim}

\DeclareMathOperator{\Quad}{Quad}
\DeclareMathOperator{\dimQuad}{dimQuad}

\newcommand{\al}{\alpha}
\newcommand{\be}{\beta}
\newcommand{\ga}{\gamma}

\newcommand{\Ga}{\Gamma}
\newcommand{\De}{\Delta}
\newcommand{\Si}{\Sigma}
\newcommand{\si}{\sigma}

\newcommand{\fie}{\varphi}
\newcommand{\Up}{\Upsilon}

\numberwithin{equation}{section}

\title{On the canonical map of surfaces with $q\ge 6$}
\author{Margarida Mendes Lopes, Rita Pardini and Gian Pietro Pirola}

\begin{document}
\begin{abstract}   We carry out an analysis of   the canonical system  of a  minimal complex surface  $S$ of general type with irregularity $q>0$. Using this analysis we are able to sharpen in the case $q>0$ the well known Castelnuovo inequality $K^2_S\ge 3p_g(S)+q(S)-7$.

Then we turn to the study of surfaces with $p_g=2q-3$ and no fibration onto a curve of genus $>1$. We prove that for $q\ge 6$ the canonical map is birational. Combining this result with the analysis of the canonical system, we also  prove the inequality:
$ K^2_S\ge 7\chi(S)+2$. This improves an earlier result of the  first and second author (\cite{mr}).
\medskip

\end{abstract}
\maketitle
{\em Dedicated to Fabrizio Catanese on the occasion of his $60^{th}$ birthday.}
\tableofcontents

\section{Introduction}

 Complex surfaces of general type  have been an object of study since the $19^{th}$ century and nowadays  their general  behaviour is believed by many  to be understood,  but in fact there  are  still   many open problems. In particular, little  is known about  the  irregular surfaces, namely the surfaces that have non zero global holomorphic  $1$-forms.  This is in part due to the fact that a fundamental  tool for the study of surfaces of general type is the canonical map, which is easier to understand in the case of  regular surfaces.

Here  (cf. \S 4)  we carry out  an analysis of the canonical system of irregular surfaces, paying special attention to the case of surfaces without  irrational pencils of genus $>1$, namely  surfaces that have no fibration onto a curve of genus $>1$, and to the case where the canonical system has a fixed part.  These results enable us to sharpen in the case of irregular surfaces the well known Castelnuovo inequality $K^2\ge 3p_g+q-7$ for a minimal surface with birational canonical map (cf. \S 5).

Next, we turn  to surfaces with $p_g=2q-3$. Recall that by \cite{appendix}  an irregular surface $S$ of general type has $p_g\ge 2q-4$,  with equality holding if and only if $S$ is birational to the product of a curve of genus $2$ and a curve of genus $q-2$. Hence it seems natural to try to classify surfaces $S$ with $p_g=2q-3$. This is easily done under the assumption that $S$ has an irrational pencil of genus $>1$ (cf. \cite{mr}, \cite{bnp}) but the matter becomes very hard if one assumes that $S$ has no such pencil.   
Examples of surfaces with these properties are known only for $q=3,4$. For $q=3$ one has the symmetric product of a curve of genus 3  and it is known (\cite{HaconPardini},\cite{Pirola}) that this is the only  surface with  $p_g=q=3$  and no irrational pencil of genus $>1$. A family of examples with $q=4$ (hence $p_g=5$) has been constructed by C. Schoen in \cite{schoen}.
 Hence one is led to doubt of the existence of these surfaces for $q\ge 5$. Indeed, in \cite {mrp} it is shown that the case $q=5$, $p_g=7$  does not occur. However the arguments used in  \cite{mrp}, besides being quite intricate, are very ad hoc and for $q\ne 3, 5$ only some general restrictions are known. For $q\ge 4$, surfaces with $p_g=2q-3$ and no irrational pencil of genus $>1$ are  ``generalized Lagrangian'', namely  they have  independent global $1$-forms $\al_1,\dots,\al_4$ such that $\al_1\wedge \al_2+\al_3\wedge\al_4=0$. In \cite{bnp} it is shown that a  minimal generalized Lagrangian surface whose canonical system has no fixed part has $K^2\ge 8\chi$ and in \cite{mr} the weaker inequality $K^2\ge 7\chi-1$ has been proven  for all surfaces  with $p_g=2q-3$.

 Here we prove that  the canonical map of surfaces with $p_g=2q-3$ that have no irrational pencil of genus $>1$ is birational  (Theorem  \ref{thm:degreecan}). Combining this result with the improved version of the Castelnuovo inequality given in \S 5, we sharpen  the inequality of \cite{mr} to $K^2\ge 7\chi+2$. It is our hope that these results are a step towards deciding in general of the existence of surfaces with $p_g=2q-3$ and no irrational pencil of genus $>1$.

The paper is organized as follows.
In \S  2 we recall  several well known  technical results that are used repeatedly in the paper.
Sections 3 and 4 are the technical heart of the paper. 
In \S 3 we establish the existence of pencils of low degree on some rational surfaces, refining similar results by Reid and Xiao (\cite{miles},\cite{xiaohigh}).
 (This result is essential for proving  
 Theorem  \ref{thm:degreecan}).
Section 4 starts with some results on  the existence of certain types of curves on an irregular surface, that are, we believe, of independent interest. Then, in order to establish the afore mentioned sharpenings of Castelnuovo's theorem,   we  study the quadrics through the canonical image of an irregular surface
and, in addition,  we give a small  refinement of an inequality due to Debarre.
In \S 5 we use the results of \S 4 to prove the Castelnuovo type inequalities. Section 6  presents the results  on surfaces with $p_g=2q-3$ and $q\geq 6$. Whilst the birationality of the canonical map for such surfaces with $q\geq 7$ is an almost immediate consequence of the results of \S 3, some more work is needed to show birationality for $q=6$. 
\smallskip

\noindent {\em Acknowledgments:\/} This research was partially supported by  FCT (Portugal) through program POCTI/FEDER and Project 
PTDC/MAT/099275/2008 and by MIUR (Italy) through project   PRIN 2007 \emph{``Spazi di moduli e teorie di Lie''}.
 The first author is a member of the Center for Mathematical
Analysis, Geometry and Dynamical Systems (IST/UTL)  and the second and the third author are members of G.N.S.A.G.A.-I.N.d.A.M..
\smallskip

\noindent{\bf Notation and conventions:}   All varieties are complex projective. A rational map $f\colon  X\to Y$ is {\em composed with a pencil} if the image of $f$ is a curve. A linear system $|D|$ on $X$ is composed with a pencil if the map given by $|D|$ is. A surface $S$ has an {\em irrational pencil of genus $b>0$} if there exists a fibration $f\colon S\to B$, where $B$ is a curve of genus $b$. If $\Si$ is a singular surface we denote by $p_g(\Si)$ and $q(\Si)$  the geometric genus  and the irregularity of a  desingularization of $\Si$.

Usually a {\em  curve} on a surface $S$  will mean  an effective non zero divisor. We denote by $\omega_C$ the dualizing sheaf $\OO_C(K_S+C)$ of a curve $C$ of $S$. 
A {\em $(-2)$-curve} on $S$ is an effective divisor $Z$ such that $Z^2=-2$ and every irreducible component  $\theta$ of $Z$ satisfies $\theta^2=-2$ and $K_S\theta=0$.  A  $(-2)$-curve $Z$  is  called a {\em $(-2)$-cycle}  if $\theta Z\leq 0$ for every component $\theta$ of $Z$, i.e. if $Z$ is the fundamental cycle of an  A-D-E singularity  in the terminology  of \cite[Ch. III,\S 3]{bpv} or the numerical cycle of a Du Val singularity  in the terminology of \cite[Ch. IV]{miles2}.

If $Y$ is a connected subset of an abelian variety $A$, we denote  by $<Y>$ the abelian subvariety of $A$  generated by $Y$. We denote by $\albdim(X)$ the {\em Albanese dimension} of a variety $X$, namely the dimension of the image of the Albanese map of $X$.

\section {Auxiliary results}
In this section we collect several technical facts that will be used repeatedly in some of the proofs. Here ``surface'' means ``smooth complex projective surface''.
\subsection {Corollaries of the index theorem}

We recall the following corollary of the Hodge index theorem:
\begin{thm} \label{Hodge}{\rm (see, e.g., \cite{bpv})}
Let $D,E$ be $\Q$-divisors on the surface $S$. If $D^2>0$ and $DE=0$
then $E^2\leq 0$ and $E^2=0$ if and only if $E$ is homologous to $0$
in rational homology.

\end{thm}

We will use mainly the following variations  of  Theorem \ref{Hodge}:
\begin{cor}\label{Hodge1}
Let $S$ be a surface and $D$ a $\Q$-divisor such that
$D^2>0$. Then  for any $\Q$-divisor  $Z$, $D^2Z^2-(DZ)^2\leq 0$.
\end{cor}

\begin{cor}\label{Hodge2}
Let $S$ be a surface and $D$ a $\Q$-divisor of $S$ such that
$D^2>0$. Then for any decomposition of $D$ as $D=A+B$ where $A,B$ are  $\Q$-divisors,
$A^2B^2-(AB)^2\leq 0$ and if equality holds then there exist $m,n\in
\Q$ such that $mA$ is homologous to $ nB$ in rational homology.
\end{cor}

\medskip
\subsection {Properties of $m$-connected curves}

We recall  that by  a {\it curve} we mean an effective non zero divisor on  a surface  and that a curve 
$D$ is {\it m-connected} if $AB\geq m$ for any decomposition $D=A+B$ with $A,B>0$. Here we list several properties related  to this notion (cf. \cite[3.9]{miles}).
\medskip

\begin{prop}[see, e.g.,  Corollary A.2 of  \cite{cfm}, also \S 3.9 of  \cite{miles2}]\label{h0D} If $D$ is a $1$-connected curve then $h^0(D,\OO_D)=1$.
\end{prop}

\begin{lem}\label{tail} Let  $S$ be minimal of general type with $K^2_S>1$ and let  $E$ be an effective divisor of $S$ such  that $E^2=-1$ and $K_SE=1$. Then $E$ is  $1$-connected, $h^0(E, \omega_E)=1$ and
$h^0(E, K_S)=1$.
\end{lem}
\begin{proof}
Suppose that $E$ is not $1$-connected. Then there is a decomposition $E=A+B$ with $A,B>0$ and $AB\leq 0$. Since $A^2+2AB+B^2=E^2=-1$ we have $A^2+B^2\geq -1$ and therefore $A^2\geq 0$ or $B^2\geq 0$.

On the other hand, since $K_S$ is nef,  for any $0<C<E$ one has $K_SC=0$ or $K_SC=1$. If $K_SC=0$ then by the index theorem $C^2<0$. If $K_SC=1$ again by the index theorem (Corollary \ref{Hodge2})  $C^2\leq 0$ and by the adjunction formula $C^2$ is odd and so $C^2\leq -1$.
So we have a contradiction, that shows that $E$ is $1$-connected.

For the  second assertion it suffices to use  that,  by the $1$-connectedness of $E$ and Proposition \ref{h0D}, $h^0(E,\OO_E)=1$, and  that $p_a(E)=1$.

 For the last  assertion note first that, since $p_a(E)=1$ and $K_SE=1$, by the Riemann-Roch theorem one has  
 $h^0(E, K_S)=1+h^1(E,K_S)$.

 Since $\omega_E=(K_S+E)|E$, by duality one has $h^1(E, K_S)=h^0(E,E)$.

 Suppose that $h^0(E,E)\neq 0$.
 
 If $E$ is irreducible, we have immediately a contradiction because  $E^2=-1$.
 
 If  $E$ is not irreducible  there is some component $\theta$ of $E$ such that $\theta E<0$. Then, if  $h^0(E,E)\neq 0$, by  \cite[Lemma (A.1)]{cfm}
 there is a decomposition $E=A+B$, with $A,B>0$ where $EA\geq BA$. Since $EA=A^2+AB$, we obtain $A^2\geq 0$, a contradiction, because we saw above that every curve $C<E$ satisfies $C^2<0$. 
 
 Thus $h^0(E,E)=0$ and $h^0(E,K_S)=1$.
\end{proof}
\begin{prop}[\cite{adjoint} Lemma 2.6, also \cite{miles2} \S 3.9]  \label{1con}   
Let $D$ be a curve on a surface $S$ such that  $D^{2} \geq 1$ and  $D$
 is nef. Then every $D' \in |D|$ is 1-connected.\par

 Furthermore if $D '=A+B$ is a decomposition of $D$ with $A$, $B$ curves such that  $A
B = 1$,  only the following possibilities can occur:
\begin{itemize} 
 \item  
$A^{2} =-1$ or $B^{2} =-1$;
\item  $A^{2} =0$ or $B^{2} =0$; \item 
$A^{2}=B^{2}=1$, $A$ and $B$ are homologous in rational homology  and $D^{2}=4$.
\end{itemize}

Also if  $D^2\geq 10$,  and $D '=A+B$ is a decomposition of $D'$ with $A$, $B$ curves such that  $A
B = 2$, only the following possibilities can occur:
\begin{itemize} 
 \item  
$A^{2} =-2$ or $B^{2} =2$;
\item  $A^{2} =-1$ or $B^{2} =-1$;
\item  $A^{2} =0$ or $B^{2} =0$.
\end{itemize}

\end{prop}
\begin{lem}[\cite{cfm} Lemma A.4)]\label{always}  
 Let $D$ be an
$m$-connected curve of a  surface $S$ and let $D=D_1+D_2$ with $D_1$, $D_2$
curves. Then, with $[p/2]$ being the integer  part of an integer $p$:
\begin{enumerate}
\item  if  $D_1  D_2=m$, then $D_1$ and $D_2$ are
$[(m+1)/2]$-connected;
\item  if $D_1$ is chosen to be minimal subject to the
condition $D_1 (D-D_1) = m$, then $D_1$ is $[(m+3)/2]$-connected.
\end{enumerate}
\end{lem}

The following  immediate consequence of  Lemma \ref{always}  and of the 2-connectdness of the canonical divisors on minimal surfaces will be used repeatedly.

\begin{cor}\label{alwaysc}
  If a canonical divisor on a minimal surface  $S$ decomposes as $K_S=A+B$ where $A,B>0$  and $AB=2$, then both $A$ and $B$ are 1-connected.
 
\end{cor}
\section{Rational surfaces of small degree}

The existence of pencils of low degree on ruled surfaces has been studied by M. Reid  (\cite{miles}) and Xiao Gang (\cite{xiaohigh}).  
In this section we prove the following  refinement of their results,   which is crucial in proving  Theorem \ref{thm:degreecan}:

\begin{thm}  \label{thm:degree}  
Let  $\Sigma\subset \pp^n$ be a rational surface  of degree $m$  not contained in any hyperplane and let  $\eta\colon \Upsilon\to \Sigma$ be the minimal desingularization. If the linear system $|H|:=\eta^*|\OO_{\pp^n}(1)|$ is complete, then:
\begin{enumerate}
\item if $n\ge9$ and $m\le \frac{3}{2} n$, then 
$\Sigma$ has a pencil of curves $|L|$  such
that every curve of $|L|$ spans at most a $\pp^r$ with $r<\frac{1}{2}n$;
\item if $n=8$, then   $\Si$ has a pencil of curves $|L|$ such 
that every curve of $|L|$ spans at most a $\pp^3$ for  $m\le 10$ and  it  has a pencil of curves of degree $\le 4$ for $m=11,12$.
\end{enumerate}
\end{thm}
\begin{proof}  
 The proof, although long, is based on the simple classical idea of  ``termination of adjunction'' on a rational surface. One considers the adjoint system $|D|:=|K_{\Upsilon}+H|$. If $\dim|D|\le 0$, then the result follows by the classification of projective surfaces of very small degree. If $|D|$ is composed with a pencil $|L|$, then  the image of $|L|$ in $\Si$ is a pencil of degree $<\frac{n}{2}$. If the system $|D|$ maps $\Upsilon$ onto a surface, then one repeats the argument  considering the second adjoint system $|K_{\Up}+D|$. Termination of adjunction means that this process eventually stops  (in our case, it actually stops at most at the second step).

In the proof we also make repeated use of the elementary  fact that a connected curve of degree $r$ spans at most a $\pp^r$. 
 \medskip
 
By \cite[Corollary 1.1]{miles}, if $4m<6n-81/4$, i.e., if $m<\frac{3}{2}n-5-1/16$, then $\Sigma$ has a pencil of lines or conics and so the theorem is true in this case.
So we are left with studying $\frac{3}{2}n-5\leq m\leq  \frac{3}{2} n$.

Write $m=\frac{3}{2}n-\alpha$.
 If $m=n-1$, then $\Sigma$  is either a cone over a rational normal curve of degree $n-1$ or it is a rational normal scroll. In either case, it   has a pencil of lines. Similarly,  if $m=n$ there are two possibilities (see  \cite{nagata}):
 \begin{itemize}
 \item[(a)] $n=8$ and $\Si$ is the  anticanonical image of $\pp^2$ blown up at a point $P$  or of a (possibly singular) quadric of $\pp^3$. In either case, $\Si$ has a pencil of conics, corresponding  in the former case to  the lines through $P$ and in the latter case to  the lines of a ruling of the quadric.
 \item[(b)] $n=9$ and $\Si$ is the anticanonical image of $\pp^2$. In this case $\Si$ has a $2$-dimensional system of curves of degree 3, the images of the lines of $\pp^2$.
 \end{itemize}
So we can assume that  $m>n$, i.e. $\frac{1}{2} n>\alpha$.

Let $H\in |H|$ be general. The curve $H$  is smooth and irreducible by Bertini's theorem and,   by the regularity of $\Upsilon$,  the system $|H|_H$ is  complete.  Since  $|H|_H$ has dimension $n-1$ and degree  $<2(n-1)$, it is  not  special by Clifford's theorem.  So restricting $\OO_{\Upsilon}(H)$ to $H$ and taking cohomology we get $h^1(\OO_{\Upsilon}(H))=0$.  Riemann-Roch  applied to $\OO_H(H)$ gives $n=\frac{3}{2}n-\alpha+1-g(H)$, namely $g(H)-1=\frac{1}{2}n-\alpha$. The adjunction formula gives   $K_\Upsilon H=-\frac{1}{2}n-\alpha$.   
 
 We consider now the adjoint linear system  $|D|:=|K_\Upsilon+H|$.  Using the adjunction sequence for $H$,  one sees that $h^0(D)= g(H)=\frac{1}{2}n-\alpha+1$ and, because we are assuming $\frac{1}{2} n>\alpha$, $h^0(D)\geq 2$. Write $|D|=Z+|M|$, where  $Z$ is the fixed part of $|D|$  and $|M|$ is the moving part.  
 \smallskip

 \noindent{\bf Step 1:} {\em $D$ is nef. In particular, we have $D^2\ge 0$.}\newline
 Since $q(\Up)=0$,  the restriction of $|D|$ to a curve $H\in|H|$ is the complete canonical system $|K_H|$. Since for a general $H$ the system    $|K_H|$ is  base point free,  for any irreducible component $\theta$ of $Z$ we have $\theta H=0$ and so, by the index theorem, $\theta^2<0$. 
 
 Let $\theta$ be an irreducible  curve such that $\theta D<0$. Since $D$ is effective, $\theta$ is a component of $Z$. Hence $\theta H=0$, $\theta K_{\Up}<0$, $\theta^2<0$, namely $\theta$ is a $-1-$curve contracted by $|H|$, against the assumption that $\Up\to \Si$ is the minimal desingularization.
 \medskip

 \noindent{\bf Step 2:} {\em If $|D|$ is composed with a pencil, then $\Sigma$ has a pencil  of conics.}\newline
 If $|D|$ is composed with a pencil  we can write $|D|=Z+|(\frac{1}{2}n-\alpha)G|$, where  $|G|$ is a pencil. Since $HZ=0$ (cf. Step 1) and $HD=n-2\alpha$, one has $HG= 2$ and the general $G$ is mapped by $\eta$  to a conic of $\pp^n$.  \medskip
  
  \noindent{\bf Step 3:} {\em If $|D|$ is not composed with a pencil, then  $D^2\geq \frac{1}{2}n-\alpha-1\ge 1$.
  Furthermore, if $D^2\leq \frac{1}{2}n-\alpha$, $|D|$ is base point free.}\newline
 Since we are assuming that $|D|$ is not composed with a pencil, we have $h^0(D)= \frac{1}{2}n-\alpha+1\ge 3$,  the general $M\in |M|$ is irreducible  and  $M^2\geq \frac{1}{2}n-\alpha-1$. The last inequality  holds because  the image of $\Upsilon$ via the map defined by $|D|$ is a non degenerate surface.

  Since $D$ is nef by Step 1,  we have $D^2\geq DM=M^2+MZ\ge M^2$ and so $D^2\geq \frac{1}{2}n-\alpha-1$.
Since $D$ is nef and $D^2>0$, every curve of $|D|$ is 1-connected by Proposition \ref{1con} , and so  $Z\neq 0$  iff $MZ>0$. Because $HM=HD=H(K_{\Upsilon}+H)$ is even (recall $HZ=0$, cf. Step 1), we obtain $M^2+K_\Upsilon M= M^2+DM-HM=2M^2+MZ-HM\equiv MZ \mod 2$. So, if $Z\ne 0$, then  $MZ\ge 2$ and $D^2\ge  \frac{1}{2}n-\alpha+1$.
 
If $D^2=   \frac{1}{2}n-\alpha-1$, then of course $|D|$ has no base points, whilst if $D^2= \frac{1}{2}n-\alpha$, $|D|$ can  have one simple base point. If  this is the case, the system  $|D|$ maps $\Upsilon$ birationally onto a surface of minimal degree in $\pp^{\frac{1}{2}n-\alpha}$. Hence the image of a general $D\in |D|$ is a rational normal curve in $\pp^{\frac{1}{2}n-\alpha-1}$. Since $D$ is smooth by Bertini's theorem, it is isomorphic to $\pp^1$. On the other hand, the restricted system $|D|_D$ has positive dimension,  it is complete since $\Upsilon$  is regular and  it has a base point since $|D|$ has one.  Since this contradicts the theory of complete linear systems on $\pp^1$, we have proven that $|D|$ has no base point for $D^2= \frac{1}{2}n-\alpha$.
 \medskip

In view of Step 2, we may assume that $|D|$ is not composed with a pencil. We finish the proof   by a case by case study, considering the various possibilities for $D^2$.  

    \noindent{\bf Step 4:} {\em The case  $D^2=\frac{1}{2}n-\alpha-1$.}\newline
Suppose first that  $\frac{1}{2} n-\alpha-1\ge 2$.
 By Step 3,  $|D|$ has no base points and maps $\Upsilon$  birationally onto a non degenerate surface  $T$ of minimal degre. There are two possibilities:
   
   (a) $T$ is ruled by lines; or
   
   (b) $T$ is the Veronese surface in $\pp^5$.
  
 In case (a), denote by $|G|$ the moving part of the pull back to $\Upsilon$ of a pencil of lines.  Since $DG=1$,  the index theorem gives   $G^2=0$. It follows $K_\Upsilon G=-2$,  $HG=3$ and therefore the curves of $|G|$ are mapped to cubics by $\eta$.

In case (b),  we have $\frac{1}{2} n-\alpha=5$, $HD=10$ and we can write $D=2\Delta$, where $\Delta$ is the pull back of a conic contained in $T$. Hence $H\Delta=5$.  This is enough to prove the statement if $n\ge 11$, namely if $\alpha\ge \frac{1}{2}$.

 If $\alpha=0$, then $n=10$, $H^2=15$. Since  $4=D^2=4\Delta^2$, we have  $\Delta^2=1$ and $|\Delta|$ gives a birational morphism to $\pp^2$. Since $K_{\Upsilon}H=-5$ and $D^2=(K_{\Up}+H)^2=4$, we get $K^2_{\Upsilon}=-1$.   Hence the morphism  $\Up\to \pp^2$ given by    $|\Delta|$ is the  composition of blow ups at ten  (possibly not distinct) points $P_1,\dots P_{10}$ of $\pp^2$. Denote by $E_1,\dots E_{10}$ the corresponding $-1$-curves of $\Upsilon$.   Then $K_{\Upsilon}=-3\De+\sum_i E_i$ and $H=D-K_{\Upsilon} =5\De-\sum_i E_i$.  The pull-back of the pencil of lines through, say,  $P_1$ gives a pencil  $|L|$ on $\Upsilon$ such that   $HL=4<5=\frac{1}{2} n$.
  
  Suppose finally  that $\frac{1}{2} n-\alpha-1=1$, namely that $|D|$ gives a birational morphism  $\Upsilon\to \pp^2$. Since $HD=4$,  the theorem is proven for $n\geq 9$ . 
  
  Suppose $n=8$, hence $\alpha=2$, $H^2=10$,  $K_\Upsilon H=-6$ and $K^2_{\Upsilon}=3$. The birational morphism $\Si\to \pp^2$ given by $|D|$  is 
the composition of  blow ups at six (possibly infinitely near) points. So $|H|=|D-K_{\Up}|$ is the pull back  of the  system of plane quartics  through these six points.  The pull-back of the pencil of lines through one of these points gives a pencil $|L|$ on $\Upsilon$ such that   $HL=3<4=\frac{1}{2} n$.
   \medskip
   
 \noindent{\bf Step 5:} {\em The case  $D^2=\frac{1}{2}n-\alpha$.}\newline
By Step 3, the system  $|D|$ has no base points and maps $\Upsilon$ birationally onto a rational surface $T$ of degree $p$ in $\pp^p$, where $p=\frac{1}{2}n-\alpha$. Since the system $|D|$ is complete, we have $p\le 9$ and $T$ is a weak Del Pezzo surface. Let $\tilde{T}\to T$ be the minimal desingularization; then $\tilde{T}$ is either  an irreducible  quadric of $\pp^3$ ($p=8$)  an irreducible  quadric of $\pp^3$ ($p=8$)  or the blow up of $\pp^2$ at $9-p$ base points,  and the map $\tilde{T}\to T\subset \pp^P$ is given by the anticanonical system $|-K_{\tilde{T}}|$. The morphism $\Upsilon\to T$ factors through a morphism $f\colon \Upsilon \to \tilde{T}$ such that $D=f^*(-K_{\tilde{T}})$.
For  $3\leq p\leq 8$ the surface $\tilde{T}$ has a  pencil of rational curves $|G|$  of degree $2$ with $G^2=0$, given  in the former case by a ruling of the quadric and in the latter case  by the lines through one of the blown-up points. Pulling back  this pencil to $\Upsilon$ we obtain a linear system $|L|$ such that  $HL=4$. This proves the theorem for $n\ge 9$. For  $n=8$ there are two possibilities, $p=4$, $\alpha=0$, $m=12$ and $p=3$, $\alpha=1$, $m=11$, which correspond to the exceptions given  in statement (b). 

If $p=9$, then the pull back of the system of  lines of $\pp^2$ gives a linear system $|L|$ such that $HL=6$. In this case we have $\frac{1}{2}n-\alpha=9$ and so $n\geq 18$. 

If $p=2$, then  $HD=4$ and so if $n\geq 9$  the assertion is proven.
We claim that $n=8$, $p=2$ does not occur.  In fact if $n=8$, then from $\frac{1}{2}n-\alpha=2$ we obtain $\alpha=2$, $H^2=10$ and $K_\Upsilon H=-6$. From $D^2=(K_\Upsilon+H)^2=2$ we obtain $K_\Upsilon^2=4$ and so $K_\Upsilon^2H^2-(K_\Upsilon H)=40-36>0$, contradicting the index theorem.
   \medskip

   \noindent{\bf Step 6:} {\em The case  $\alpha\geq 1$ and $D^2\geq \frac{1}{2}n-\alpha+1$.}\newline
  By $D^2=K_\Upsilon^2+\frac{1}{2}n-3\alpha$, in this case  $K_\Upsilon^2-2\alpha\geq 1$.  By the Riemann-Roch theorem and Kawamata-Viehweg vanishing   $h^0(K_\Upsilon+D)= K_\Upsilon^2-2\alpha+1$ and so $h^0(K_\Upsilon+D)\geq 2$. On the other hand $H(K_\Upsilon+D)= \frac{1}{2}n-3\alpha$.   Hence a general curve $L$  in the moving part of $|K_\Upsilon+D|$ satisfies $HL\le \frac{1}{2}n-3\alpha<\frac{1}{2} n$.
  \medskip
    
   \noindent{\bf Step 7:} {\em The case  $\alpha=0$ and   $D^2\geq \frac{1}{2}n+1$.}\newline
  In this case $D^2\geq \frac{1}{2}n+1$ implies that $K_\Upsilon^2\geq 1$, because $D^2=K_\Upsilon^2+\frac{1}{2}n$.  Hence $h^0(-K_\Upsilon)\geq 2$. 
   
 Let $C$ be a curve in the moving part of $|-\!K_{\Upsilon}|$. 
   The kernel of the restriction map   $H^0(H)\to H^0(C, H|_C)$ is $H^0(H-C)$.  One has  $ h^0( H-C)\ge h^0(H+K_\Upsilon)=\frac{1}{2}n+1$ and  $h^0(H)=n+1$.  We conclude that the image via $|H|$ of 
 $C$  spans a projective space of dimension $<\frac{1}{2} n$.
  \smallskip
  
  So,  having covered all possible cases,  we have proven the theorem.
   \end{proof}

\bigskip
For later reference  we examine more closely   one of the exceptions in case (ii) of Theorem  \ref{thm:degree}    .  

\begin{prop} \label{deg11} Let  $\Sigma\subset \pp^8$ be a rational surface  of degree $11$  not contained in any hyperplane and let  $\eta\colon \Upsilon\to \Sigma$ be the minimal desingularization. If the linear system $|H|:=\eta^*|\OO_{\pp^n}(1)|$ is complete and $\Sigma $ has no pencil of curves of degree $<4$, then $H$ decomposes as $H=2H' + J$, where  $J$ is a non zero effective divisor, $h^0(\Sigma, H')\geq 3$ and the linear system $|H'|$ has no fixed components.
\end{prop}
\begin{proof}
A surface satisfying the hypothesis is as in  Step 5 of proof of   theorem \ref{thm:degree}. 
By the proof  and keeping the same notation, one has that a surface   
of degree 11 in $\pp^8$  that has no pencil of curves of degree $<4$ satisfies $D^2= 3$, $K_\Upsilon D=-3$ and $K_\Upsilon^2=2$.   From this we have that $h^0(K_\Upsilon+D)\neq 0$ and that $K_\Upsilon+D\neq 0$. 
Since $\Upsilon $ is rational  and $K_\Upsilon^2=2$, $h^0(-K_\Upsilon)\geq 3$.  Then  $H$ decomposes as  $H=(-2K_\Upsilon)+ (K_\Upsilon+D)$  and taking the moving  part of 
$|-K_\Upsilon|$ we have the statement.
\end{proof}

\section{Irregular surfaces }\label{sec:irreg_surf}

In this section we collect several technical results that are needed in \S \ref{sec:castelnuovo}   and in \S \ref{sec:2q-3}, but are also, we believe, of independent interest. 

Throughout all the section we denote by  $S$  a smooth projective irregular surface, by $q>0$ the irregularity of $S$   and by  $a\colon S\to A:=\Alb(S)$ the Albanese map.

\subsection{Curves on irregular surfaces without irrational pencils}
\begin{lem}\label{lem:pa(D)} If  $D$ is an effective  $1$-connected divisor of $S$, then 
$<\!a(D)\!>$ has dimension $\le p_a(D)$. 
\end{lem}
\begin{proof} Write $D=D_{red}+A$, where $D_{red}$ is the support of $D$ and $A\ge 0$.  It is easy to show that $D_{red}$ is connected. Moreover, if  $A>0$ then the decomposition sequence:
$$0\to \OO_A(-D_{red})\to \OO_D\to \OO_{D_{red}}\to 0$$
shows that $p_a(D)=h^1(\OO_D)\ge h^1(\OO_{D_{red}})=p_a(D_{red})$. Hence we may assume that $D$ is reduced.

We prove the statement by induction on the number $n$ of irreducible components of $D$. If $n=1$, then there is a surjective morphism $$J\to <a(D)>$$ \noindent where $J$ is the Jacobian of the normalization of $D$. Since $J$ has dimension $g(D)\le p_a(D)$, the statement follows.

To prove the inductive step, write $D=C+D_1$, where $C$ is an irreducible curve and $D_1$ is a connected effective divisor with $n-1$ components. Since $D_1$ is connected,  the decomposition sequence gives an exact sequence:
$$0\to H^1(\OO_C(-D_1))\to H^1(\OO_D)\to H^1(\OO_{D_1})\to 0.$$
Thus we have $p_a(D)=p_a(D_1)+h^1(\OO_C(-D_1))\ge p_a(D_1)+p_a(D)$. To complete the proof it is enough to notice that $<\!a(D)\!>=<\!a(D_1)\!>+<\!a(C)\!>$.
\end{proof}
The next  lemma is a generalization of \cite[Proposition 8.2, (a)]{bnp}.

\begin{lem}\label{lem:D2}  Let $S$ be a surface such that  $\albdim(S)=2$ and   let $D>0$ be    a divisor of $S$ such that one of the  following conditions holds:
\begin{itemize}
\item $D$ is irreducible and $g(D)<q$;
\item  $D$ is $1$-connected and $p_a(D)<q$.
\end{itemize}
Then:
\begin{enumerate}
\item $D^2\le 0$;
\item  if $D^2=0$,   then there is a fibration $f\colon S\to B$, where $B$ is a curve of genus  at least $q-p_a(D)$,  and there exists a integer $m>0$ such that   $mD$ is a fibre of $f$.
\end{enumerate}
\end{lem}

\begin{proof}  Consider the abelian variety $A':=A/\!\!<\!\!a(D)\!\!>$ and denote by $a'\colon S\to A'$ the map induced by $a$. By  the assumptions and by Lemma \ref{lem:pa(D)}, $\dim A'>0$ and  the image $Z$  of $a'$ generates $A'$ by construction.

Assume that $Z$ is a surface and let $H$ be the pull back to $S$ of a very ample line bundle of $Z$. Then $HD=0$, $H^2>0$, hence  $D^2<0$ by the index theorem. So  if $D^2\ge 0$ then $Z$ is a curve and $a'$ is composed with a pencil $f\colon S\to B$, where $B$ is a smooth curve. Since $B$ maps onto $Z$ and $Z$ generates $A'$, we have $g(B)\ge \dim A'\ge q-p_a(D)$.
Since by construction $a'$ contracts $D$ to a point and $D$ is connected, $D$ is contained in a fiber of $f$. By Zariski's lemma one has $D^2\le 0$ and $D^2=0$ if and only if there exists an integer $m>0$ such that $mD$ is a fiber of $f$. 
 \end{proof}
 

\begin{cor} \label{f1} Let $D>0$  be a $1$-connected divisor of $S$ such that $D^2=0$. If $b\ge 0$ is an integer such that  $S$ has no irrational pencil of genus $>b$, then:
$$K_SD\geq 2(q-b)-2.$$ \end{cor}

 \subsection{Some properties of the canonical system}
In this section we assume that the canonical system $|K_S|\ne \emptyset$. We write  $p_g:=p_g(S)$ and $|K_S|=|M|+Z$, where $|M|$ is the moving part and $Z$ is the fixed part. We  denote by $\Sigma$ the canonical image and by $\fie \colon S\to \Si\subset \pp^{p_g-1}$ the canonical map. 

\begin{lem}[\cite{severi}, Lemma 2.1] \label{pg}  Let $\iota$ be an involution of $S$ such that $p_g(S/\iota)=0$. If $q\ge 3$ then $S$ has an irrational pencil $f\colon S\to B$,  where $g(B)\geq 2$.
\end{lem}

  \begin{cor}\label{ruled} Let $S$ be a minimal surface with $q\ge 3$. If $S$  has no irrational pencil of genus $\ge 2$   and  $\Si$ is    a surface   with $p_g(\Sigma)=0$,  then  $q(\Sigma)\leq 1$ and $\deg\fie\geq 3$.   
  \end{cor} 
  
   \begin{proof}   If  $q(\Sigma)\geq 2$ then,  by the  classification of surfaces,  $\Sigma$ must be a ruled surface  and this is a contradiction because the pull-back of the ruling of $\Sigma$ would give an irrational pencil with base of  genus $\geq 2$.  
   
   Since $p_g(\Sigma)=0$   the canonical map of $S$ is not birational and so by Lemma \ref{pg}  its degree must be $\geq 3$.  \end{proof}

The following result is essentially contained in  \cite{xiaopencil}: 

\begin{prop}\label{generation} If  $\albdim S=2$ and    $C$ is an irreducible curve of $S$,  then:
\begin{enumerate}
\item  if $h^0(S,C)=s\geq 2$, then $\fie(C)$ spans  at least a $\pp^{q-2}$ and $p_a(C)\ge 2q-3+s$. 
\item if $C$ is a general fiber of a fibration $f\colon S\to B$,  with $B$ a curve  of genus $b>0$, then $\fie(C)$ spans at least a  $\pp^{q-b-2}$.  
\end{enumerate}
\end{prop}

\begin{proof}  (i)  By \cite{xiaopencil},  if $S$ is an irregular surface of maximal Albanese   dimension and  $C$ is   a curve  of $S$ that moves in a linear system, then    the image of the restriction map $H^0(S,K_S)\to H^0(C, K_{S}|_{C}) $  has dimension   at least  $q-1$. Passing to cohomology, the adjunction sequence for $C$ gives:
$$0\to H^0(S,K_S)\overset{r}{\to} H^0(K_S+C)\to H^0(C,\omega_C)\to H^1(S,K_S)\to 0,$$
where exactness on the right follows by Ramanujam's or by Kawamata-Viehweg's vanishing. Hence we have:
$$p_a(C)=h^0(C,\omega_C)\ge h^1(S, K_S)+\dim {\rm  Im} r=q+\dim {\rm  Im} r.$$
The subspace $ {\rm  Im} r$ contains the image of $H^0(S, K_S)\otimes H^0(S,C)$, hence it has dimension $\ge (q-1)+s-2=q-3+s$.

 (ii)  Also by \cite{xiaopencil} (see \cite[Proposition 2.2]{severi}),  given a pencil $f\colon S\to B$  with general fibre $C$  and such that $g(B)=b$ the image of the restriction map $H^0(S,K_S)\to H^0(C, \omega_C )$ has dimension at least $q-b-1$.  \end{proof}

Following \cite{konno}, we define the {\em quadric hull\/} $\Quad(S)$  of a surface  of general type $S$ as the intersection of all the quadrics of $\pp^{p_g-1}$ that contain the canonical image  $\Sigma$. A component of $\Quad(S)$ is said to be {\em essential}  if it contains  $\Sigma$; the {\em quadric dimension\/} $\dimQuad(S)$ is  the maximum dimension of an essential  component of $\Quad(S)$. 
We quote the following:

\begin{prop}[\cite{cmp}, Proposition 2.4] \label{3fold}
Let $X\subset \pp^{r+1}$ be a non degenerate irreducible threefold  and let $\ga$ be the arithmetic genus of a general curve section of $X$.  Then:
\begin{enumerate}
\item if  $\ga=0$, then $X$ is either a rational normal scroll or $X\subset\pp^6$ is the cone over the Veronese surface in $\pp^5$;
\item if $\ga=1$ and $X$ is not a scroll then   $r\leq 9$;
\item if $\ga=2$  and $X$ is not a scroll  then   $r\leq 11$.
\end{enumerate}
\end{prop}

\begin{prop}\label{prop:quadric}
Let  $S$ be a  surface such that $\albdim S=2$ and  $\fie$ is birational. Then:
\begin{enumerate}
\item if $p_g\ge 8$ and $q\ge 5$, then $h^0(2M)\ge 4p_g-5$; 
\item if $p_g\ge 12$, $q\ge 6$,  then  $h^0(2M)\ge 4p_g-4$;
\item if  $p_g\ge 14$, $q\ge 7$, then  $h^0(2M)\ge 4p_g-3$.
\end{enumerate}
\end{prop}
\begin{proof}
Set $r=p_g-2$. Notice that by the Castelnuovo inequality (cf. \cite[Remarques 5.6]{beauville}) we have $\deg\Si\ge 3p_g-7\ge 3r-1$.  

It is well known (cf. \cite{deb1}, \cite{Reid_quadrics}, \cite{barjaosaka}) that $h^0(2M)\geq 4p_g-6=4r+2$. 
We argue by contradiction, writing  $h^0(2M)=4r+2+\alpha$ and assuming that  one of the following holds:
\begin{itemize}
\item $\alpha=0$,  $r\ge 6$ and $q\ge 5$; 
\item $\alpha=1$,  $r\ge 10$ and $q\ge 6$; 
\item $\alpha=2$, $r\ge 12$ and $q\ge  7$.
\end{itemize}

For a non degenerate projective variety $Y\subset \pp^{r+1}$ and $m\ge 0$ an integer, denote as usual  by $h_Y(m)$ the {\em Hilbert function} of $Y$, namely the dimension of the image of the restriction map $H^0(\OO_{\pp^{r+1}}(m))\to H^0(\OO_Y(m))$. 
In what follows we use some basic properties of the Hilbert function, for which we refer the reader to \cite{Harris_curve}.

Let $C$ be a general section of the canonical image $\Si$ and let $Z$ be a general section of $C$. The set $Z$ consists of $\deg\Si\ge 3r-1$ points in uniform position and one has:
 \begin{equation}\label{eq:hilbert}
 4r+2+\alpha =h^0(2M)\ge h_{\Sigma}(2)\ge r+2+h_C(2)\ge 2r+3 +h_Z(2),
\end{equation}
namely $h_Z(2)\le 2r-1+{\alpha}$.
\medskip

\noindent{\bf Step 1:} {$\dimQuad(S)\ge 3$}

By  \cite[Lemma 3. 9]{Harris_curve} one has $h_Z(2)\ge 2r-1$. Hence  by \eqref{eq:hilbert}, there are  the following possibilities: 
\begin{itemize}
\item[(a)] $h_Z(2)=2r-1$. By \cite[Lemma 3. 9]{Harris_curve},  in this case the intersection of all quadrics through $Z$ is  a rational normal curve in $\pp^{r-1}$;

\item[(b)] $h_Z(2)=2r$.  By \cite[p. 109]{Harris_curve}, in this case the intersection of all quadrics through $Z$ is  a rational normal elliptic curve of degree $r$ in $\pp^{r-1}$.

\item[(c)] $h_Z(2)=2r+1$. Since $p_g\ge 8$,  by   \cite [Theorem 3.8] {Ciro_curve} (cf. also \cite[Proposition 4.3]{petrakiev}), in this case the intersection of all quadrics through $Z$ is  an irreducible  curve of degree $r+1$ in $\pp^{r-1}$.
\end{itemize} 
 In each case, the intersection of all the quadrics   of $\pp^{r-1}$ containing $Z$ is an irreducible curve $\Gamma$. If $V$ is an essential component of $\Quad(S)$, then $\Quad(S)\cap \pp^{r-1}$ contains $\Gamma$. Since $\pp^{r-1}\subset \pp^{r+1}$ is a general codimension 2 subspace,  it follows that $\dim V\ge 3$.
 \medskip

 \noindent{\bf Step 2:} {\em $\Quad(S)$  has no essential component of dimension 3.}\newline
 Assume for contradiction that an essential component $V$ of $\Quad(S)$ of dimension 3 exists.  Then by the proof of Step 1,  the general curve section $\Gamma$  of $V$ has arithmetic  genus $\le \alpha$.
Hence, in view of  our  assumptions on $r$ and $\alpha$, by Proposition \ref{3fold} $V$ is a scroll in planes. If $\Gamma$ is rational, as it is always the case for $\alpha=0$, we let $|C|$ be the  pencil of $S$ induced by the ruling of $V$. Since  $q\ge 5$ by assumption, we have a contradiction to Proposition \ref{generation}.  If $\gamma$ has geometric  genus $b>0$,  we let $B\to \Ga$ be the normalization map and $f\colon S\to B$ the fibration induced by the ruling of $V$. Since $b \le \al$ and  $q\ge 5+\al$ by assumption, we have again a contradiction to Proposition \ref{generation}.
\medskip

\noindent{\bf Step 3:} {\em $\dimQuad(S)\le 3+\al$.}\newline
If  $\al=0$, (cf. also \cite{konno}),  $\Quad(S)$ is a threefold  by \cite[Proposition 1.2]{barjaosaka}. 

Consider now $\al>0$ and assume  for contradiction that $\dimQuad(S)\ge 4+\al$. Since the  quadrics through $Z$ cut  out a curve in $\pp^{r-1}$ (cf. proof of Step 1), it follows that the image of  the restriction map $\rho \colon H^0(\pp^{r+1},\mathcal I_{\Si}(2))\to H^0(\pp^{r-1}, \mathcal I_Z(2))$ is a subspace of codimension $\ge 1+\al$. Since $Z\subset C\subset \Si$ are general sections and $\Si$ is non degenerate, the sequences
 $0\to \mathcal I_{\Si}(1)\to  \mathcal I_{\Si}(2)\to\mathcal I_C(2)\to 0$
and $0\to \mathcal I_{C}(1)\to  \mathcal I_{C}(2)\to\mathcal I_Z(2)\to 0$ are exact. Taking cohomology, one sees that the restriction maps $H^0(\pp^{r+1}, \mathcal I_{\Si}(2))\to H^0(\pp^r, \mathcal I_C(2))$ and $H^0(\pp^r, \mathcal I_C(2))\to H^0(\pp^{r-1}, \mathcal I_Z(2))$ are injective.
Hence $\rho$, being the composition of these maps, is also  injective and we get  $h^0(\pp^{r-1},\mathcal I_Z(2))\ge h^0(\pp^{r+1},\mathcal I_{\Sigma}(2))+1+\al$. Passing to the Hilbert functions, we obtain:
\begin{gather}
h_{\Sigma}(2)-h_Z(2)=\\
\nonumber \frac{(r+2)(r+3)}{2}-h^0(\pp^{r+1},\mathcal I_{\Si}(2))-\frac{r(r+1)}{2}+h^0(\pp^{r-1}, \mathcal I_Z(2))
\ge 2r+4+\al.\
\end{gather}
Since $h_Z(2)\ge 2r-1$, we get  $4r+2+\al=h_{\Si}(2)\ge 4r+3+\al$, 
  a contradiction.
\medskip

\noindent{\bf Step 4:} {\em End of proof.}\newline
If $\al=0$, then we have a contradiction by Step 2 and Step 3.

If $\al=1$, then by Step 2 and Step 3 we have $\dimQuad(S)=4$.   By \cite[Lemma 1.2]{konno}, we have:
$$4r+3=h^0(2M)\ge  h_{\Si}(2)\ge 5p_g-10=5r,$$
a contradiction since $r\ge 4$.

If $\al=2$, then by Step 2 and Step 3 we have $\dimQuad(S)=4$ or $5$.   By \cite[Lemma 1.2]{konno}, we have:
$$4r+4=h^0(2M)\ge  h_{\Si}(2)\ge\min\{ 5p_g-10, 6p_g-15\}=5r,$$ 
 and we have again a contradiction since $r\ge 5$.
\end{proof}

    \begin{lem}\label{z=-1}  Let        $S$ be a minimal surface  with $q\ge 3$ and  no irrational pencil  of genus $\ge 2$ and let $D$ is a divisor of $S$ such that:
      \begin{itemize}
   \item  $D^2\geq 6$,  $h^0(D)\geq 4$ and $|D|$ has no fixed component;
 \item $F:=K_S-D>0$ and $K_SF < 2q-4$.
   \end{itemize}
    Then  for any effective divisor $E$ such that $E^2=-1$, $K_SE=1$ and $DE=2$,  $h^0(K_S+D)\geq  h^0(K_S+D-E)+3 $.  
    \end{lem} 
 \begin{proof}  By Corollary \ref{tail}, one has $h^0(E, \omega_E)=1$.   
 Using the Riemann-Roch theorem and $p_a(E)=1$ we see that  $h^0(E, D|_E)=h^1(E, D|_E)+2$.  By duality, $h^1(E,D|_E)=h^0(E, \omega_E-D|_E)$. By assumption, there exists a section $s\in H^0(S, D)$ that does not vanish on any component of $E$.  The section $s$ induces an injective map $H^0(E, \omega_E-D|_E)\to H^0(E,\omega_E)$. Hence  $h^1(E, D|_E)=h^0(E, \omega_E-D|_E)\le h^0(E,\omega_E)=1$,  and we get
  $h^0(E,D|_E)\leq 3$, $h^0(D-E)\geq 1$.  Note also that $ED=2$ implies that $EF=-1$.
  
To prove the lemma  it suffices to show that the restriction map $r\colon H^0(K_S+D)\to H^0(E, (K_S+D)|_E)$ is surjective, because by Riemann-Roch and $(K_S+D)E=3$,  $h^0(E, (K_S+D)|_E)\ge 3$. 
 
 The cokernel of $r$ is $H^1(K_S+D-E)$. Since $D^2\geq 6$, we have $(D-E)^2>0$,  hence  by Ramanujam's vanishing to prove the assertion it is enough  to show that  the effective divisor $D-E$ is $1$-connected (\cite{ramanujam}, cf. \cite{bombieri}, p.453).
 \medskip
 
Suppose for contradiction that  $D-E$ is not $1$-connected.  Then there is a decomposition $D-E=A+B$ where $A,B$ are effective non zero divisors such that $AB\leq 0$.  Since $(A+B)E=3$ and, by the $1$-connectedness of  $D$,  $A(B+E)\geq 1$, $B(A+E)\geq 1$  we must have $2AB+3\geq 2$, and so $AB= 0$. 
 From  $(A+B)E=3$ and  the $1$-connectedness of every curve in $|D|$  (Proposition \ref{1con}),  we  have, say, $AE=1$ and $BE=2$.   
 So $A(B+E)=1$.   Since,  by hypothesis $D^2\geq 6$,  Proposition \ref{1con}  tells us that $A^2\le 0$ or $(B+E)^2\le 0$.  Since $D$ is nef and $DE=2$, $2\leq D(B+E)=(B+E)^2+1$. So $A^2\le 0$  and  the nefness of $D$ implies that $A^2=DA-1\ge -1$.
    
 Since $A(B+E)=1$, the 2-connectedness of the divisors in $|K_S|$ implies that $AF>0$ and $(B+E)F>0$ and so, by the hypothesis $DF<2q-4$ we obtain  $AF<2q-5$.  Note also that, since $D$ is 1-connected,  $A$ is also 1-connected by Lemma \ref{always}. 
 
If  $A^2=0$ then $DA=1$ and  $K_SA=1+FA<2q-4$. Since $S$ has no irrational pencils of genus $>1$ this is a contradiction to  Corollary \ref{f1}.  

So  $A^2=-1$. In this case   $(A+E)^2=0$ and $(A+E)F=AF+EF<2q-6$, yielding $K_S(A+E)=(D+F)(A+E)<2+2q-6=2q-4$.  If $(A+E)$ is 1-connected we have again  a contradiction to Corollary \ref{f1}.
 
 So suppose that  $A+E$ is not 1-connected.  Then it decomposes as  $A_1+A_2$ where $A_1A_2 \leq 0$.  By 1-connectedness  of $D$ and $(A+E)B=2$ we must have $A_i(D-A_i)=1$, for $i=1,2$,  and  so we conclude as above  that $A_i$ is 1-connected and $A_i^2\leq 0$,  for $i=1,2$.  Then from $0=(A+E)^2=A_1^2+2A_1A_2+A_2^2$  one obtains   $A_1A_2=A_1^2=A_2^2=0$.  But then, since $K_S(A_1+A_2)=K_S(A+E)<2q-4$, we have  $K_SA_1<2q-4$,   contradicting   Corollary \ref{f1}. 
 
 So $D-E$ is 1-connected and therefore the Lemma is proven.
   \end{proof}
   
  We recall   the following result:
   \begin{prop}[\cite{mrp}, Corollary 2.7]\label{canonical}  Let $S$ be a minimal surface of general type whose  canonical map is not composed with a pencil.   Denote by $|M|$ the moving part and by $Z$ the fixed part of $|K_S|$. If  $Z>0$ and $M^2\geq 5+K_SZ$, then 
$$K_S^2+\chi(S)=h^0(K_S+M)+K_SZ+MZ/2 \geq h^0(2M)+K_SZ+MZ/2+1.$$

Furthermore,  if  $h^0(K_S+M)=h^0(2M)+1$ then  $|K_S+M|$ has base points and  there is an effective divisor $G$ such that  $GZ\geq 1$ and either $G^2=-1$ and $MG=0$ or $G^2=0$ and $MG=1$.
  \end{prop}
Now we can show the following

\begin{cor}\label{fixed} Let $S$ be an irregular minimal surface such that    
$S$ has  no irrational pencils $f\colon S\to B$ with $g(B)\geq 2 $, $q\geq 6$  and the canonical  map of $S$ is not composed with a  pencil.
 Denote by $|M|$ the moving part and by $Z$ the fixed part of $|K_S|$. If  $Z>0$, then 
$$K_S^2+\chi(S) \geq h^0(2M)+3.$$
 Furthermore if equality holds then $Z^2=-2$ and $K_SZ=0$.
      \end{cor}  
      
      \begin{proof}
    The hypothesis that $S$ has  no irrational pencils $f\colon S\to B$ with $g(B)\geq 2 $ implies that $p_g\geq 2q-3$. Since $q\geq 6$, we have then $p_g\geq 9$ and  so the hypothesis that the canonical  map of $S$ is not composed with a  pencil implies  that $M^2\geq 6$. 
  
       Since $K_S^2+\chi(S)=h^0(K_S+M)+K_SZ+MZ/2 $ to prove the corollary we need to study the number $m:=p+K_SZ+MZ/2$, where $p:=h^0(K_S+M)-h^0(2M)$. 
  Note that $MZ$ is an even  positive number by the 2-connectedness of the canonical divisors. Also, by Corollary \ref{alwaysc}, 
for any decomposition $K_S=A+B$ with $A,B>0$ and $AB=2$ 
  both $A$ and $B$ are 1-connected.
  
  \medskip
  We start by analyzing the case  $MZ=2$. 
   If $MZ=2$, then $Z$ is 1-connected  by Corollary \ref{alwaysc} and because $K_S$ is nef $Z^2\geq -2$.  On the other hand, by the index theorem (Corollary  \ref{Hodge1})   and $M^2\geq 6$, we have $Z^2\leq 0$.  Furthermore the hypothesis that  $S$ has  no irrational pencils $f\colon S\to B$ with $g(B)\geq 2 $  implies that $Z^2=0$ does not occur, because if $Z^2=0$ then $K_SZ=2$ and this is impossible by Corollary \ref{f1}. 
  
  So we are left with the possibilities:

  \begin{enumerate}
\item $Z^2=-1, K_SZ=1$;
\item $Z^2=-2,  K_SZ=0$.
\end{enumerate}

 Note that $K_S+M-Z=2M$.

 In the first case   Lemma \ref{z=-1} gives  $p\geq 3$ yielding $m\geq 5$.
  
  In the second case suppose that $m<3$.  By Proposition  \ref{canonical},  we see that    $p=1$, and  that   there is an effective divisor $G$ such that    $GZ\geq 1$ and either $G^2=-1$ and $MG=0$ or $G^2=0$ and $MG=1$.  It is easy to check that $MZ=2$ implies $GZ=1$.  As above  $G^2=0$ can be excluded using the hypothesis that  $S$ has  no irrational pencils $f\colon S\to B$ with $g(B)\geq 2 $.  If $G^2=-1$   we can apply  Lemma \ref{z=-1} to the divisor $E=G+Z$ and obtain that  $h^0(K_S+M)\geq  h^0(K_S+M-E)+3 $.   Since $K_S+M-E=2M-G$,  $MG=0$ and $G$ 1-connected  imply $h^0(2M-G)\geq h^0(2M)-1$ we obtain $p\geq 2$ contradicting $p=1$.
  
  So we proved the Corollary for the case $MZ=2$.
  \medskip
  
 Supose now  that $m=3$ and $MZ\geq 4$. Then  either $MZ=6$, $K_SZ=0$ and $p=0$ and this is excluded by  Corollary \ref{canonical}  or $MZ=4$ and $p+K_SZ\leq 1$. Since $M^2\geq 6$,  the case  $p=0$ and $K_SZ=1$ is again excluded by Corollary  \ref{canonical}  and so we are left with the case $MZ=4$, $K_SZ=0$ and  (by Corollary \ref{canonical}) $p=1$.  It is not difficult to verify  that $MZ=4$, $K_SZ=0$ imply that  $Z$ decomposes as $Z=Z_1+Z_2$ such that $K_SZ_i=0$, $Z_i^2=-2$ and $Z_iM=2$.   
 
  Since $p=1$,  by Corollary   \ref{canonical}  there is an effective divisor $G$ such that either $G^2=-1$ and $MG=0$ or $G^2=0$ and $MG=1$.   Again the second possibility for $G$ can be excluded as before.  In fact because $G(M-G)=1$ the 2-connectedness of the canonical divisors implies that $GZ>0$ and $(M-G)Z>0$. Since $MZ=4$, one must have $GZ\leq 3$.  If $G^2=0$ then $K_SG\leq 4$ and this is impossible by     Corollary \ref{f1}. 
  If $G^2=-1$ and  $GZ=3$ then $(G+Z)^2=1$ and $K_S(G+Z)=3$. Since $(M-G)^2>0$ and  $M^2\geq 6$ and $MZ=4$ imply $K_S^2\geq 10$, we have a contradiction to  Proposition \ref{1con}.
  So $G^2=-1$ and $GZ=1$. If $GZ_i>1$ then $(G+Z_i)^2>0$ and we find the same contradiction as above. So, say, $GZ_1=1$ and $GZ_2=0$. Then the divisor $E:=G+Z_1$ satisfies $E^2=-1$ and $K_SE=1$ and  as before applying  Lemma \ref{z=-1} we obtain $p\geq 2$, a contradiction. 
  So if $MZ=4$, $m\geq 4$.
   \end{proof}

    \section{Castelnuovo type inequalities }\label{sec:castelnuovo}

The Castelnuovo inequality (cf. \cite[Th\'eor\`eme 3.2]{deb1})  states  that if $S$ is a minimal surface of general type such that $\fie$ is birational then $K^2_S\ge 3p_g+q-7$. In the case $q>0$, in  \cite[Theorem 2.1]{barjaosaka}  the inequality has been improved to $K^2_S\ge 3p_g+q-6$ under the assumption  that $p_g\ge 6$.

By applying  the results of \S \ref{sec:irreg_surf}  we are able  to improve further the inequality in the case of surfaces with $q\ge 6$ 
 (Theorem  \ref{thm:castelnuovo}) and to sharpen it further under the assumption that $S$ has no irrational pencil and  $|K_S|$ has a  fixed part (Theorem \ref{thm:fixed}). As in the previous section, $S$ denotes a smooth  complex projective surface with geometric genus $p_g$ and irregularity $q$ and the canonical map of $S$ is denoted by $\fie\colon S\to\pp^{p_g-1}$.

\begin{thm} \label{thm:castelnuovo}
Assume that $S$ is minimal  and $\fie$ is birational. Then:  
\begin{enumerate}
\item if $q>0$ and $p_g\ge 6$, then $K^2_S\ge 3p_g+q-6$;
\item if $q\ge 6$ and $p_g\ge 12$, then $K^2_S\ge 3p_g+q-5$;
\item if $q\ge 7$ and $p_g\ge 14$, then $K^2_S\ge 3p_g+q-4$.
\end{enumerate}
\end{thm}
\begin{proof}  
Statement  (i) is \cite[Theorem 2.1]{barjaosaka}.

If  $\albdim S=1$, then  $K^2_S\ge 3p_g+7q-7$ by \cite[Theorem 6.1]{konnoslope}. Hence we may assume $\albdim S=2$.  

Since $S$ is minimal, by Riemann--Roch we have
$K^2_S+\chi(S)=h^0(2K_S)\ge h^0(2M)$, namely $K^2_S\ge h^0(2M)-p_g+q-1$. Hence  (ii) and (iii) follow directly by Proposition \ref{prop:quadric}.
  \end{proof}
 
  \begin{thm}\label{thm:fixed} Assume that $S$ is minimal with no 
   irrational pencils of genus $\ge 2$    and  that $\fie$ is birational. If the canonical system $|K_S|$ has a fixed part $Z>0$, then: 
  \begin{enumerate}
 \item   if $q\ge 6$, then $K^2_S\ge 3p_g+q-3$;
\item if $q\ge 6$ and $p_g\ge 12$, then $K^2_S\ge 3p_g+q-2$;
\item if $q\ge 7$ and $p_g\ge 14$, then $K^2_S\ge 3p_g+q-1$.
\end{enumerate}
Furthermore, if equality holds in (i), (ii) or (ii), then  $Z^2=-2$ and $K_SZ=0$.
  \end{thm}
  \begin{proof} 
   By Corollary \ref{fixed} we have $K^2_S+\chi(S)\ge h^0(2M)+3$, with equality holding only if $Z^2=-2$ and $K_SZ=0$. The result now follows immediately by Proposition \ref{prop:quadric} (notice that for $q\ge 6$ one has $p_g\ge 9$ by the Castelnuovo- De Franchis inequality).
  \end{proof}

              \section{ Surfaces with  $p_g=2q-3$}\label{sec:2q-3}
   
   Throughout all the section we  consider a minimal surface $S$  with irregularity $q$ and geometric genus $p_g(S)=2q-3$.  
   We denote by $\Si$ the canonical image and by $\fie\colon S\to\Si\subset \pp^{2q-4}$ the canonical map.
   
  The purpose of the section is to prove  the following:
  \begin{thm}\label{thm:degreecan}  If  $S$ is minimal with $q\ge 6$, $p_g=2q-3$ and  has no irregular pencil of genus $\ge 2$,  then the canonical map $\fie$ is birational.
\end{thm} 

As a consequence, we are able to strengthen the inequalities  of  \cite[Theorem 1.2]{mr} as follows:   
\begin{thm} If $S$ is minimal with  $p_g=2q-3$, then:
\begin{enumerate}
\item if $q\ge 6$, then $K^2_S\ge 7\chi(S)+2$;
\item if $q\ge 8$, then $K^2_S\ge 7\chi(S)+3$;
\item if $q\ge 9$, then $K^2_S\ge 7\chi(S)+4$.
\end{enumerate}
Furthermore if equality holds then the fixed part $Z$ of $|K_S|$ is a $(-2)$-cycle of type $D_n, E_6,E_7$ or $E_8$.
\end{thm}
 \begin{proof}
 As explained in the proof of  \cite[Thm.1.2]{mr}, we may assume that $S$ has no irrational pencil of genus $>1$ and that $|K_S|=|M|+Z$, with the fixed part $Z>0$. 

Since  in this case the canonical map $\fie$ is birational by Theorem \ref{thm:degreecan}, we get the inequalities  by applying Theorem  \ref{thm:fixed}. Again by Theorem \ref{thm:fixed}, one has equality only if $K_SZ=0$ and $Z^2=-2$.  

Note that  for every component $\theta $ of $Z$,  $M\theta \geq 0$ because $|M|$ is the moving part of $|K_S|$.   Since $K_S\theta=0$ we see that every component $\theta $ of $Z$ satisfies $\theta Z\leq 0$.
 So $Z$  is a $(-2)$-cycle and as such  it can be of  of type $A_n$, $D_n$, $E_6$, $E_7$ or $E_8$ (see,  e.g.,  \cite[Ch.III,\S 3]{bpv}). However if $Z$ is of type $A_n$ then by \cite[Theorem 5.4]{bnp}  one has $K^2_S\ge 8\chi(S)$, a contradiction.
\end{proof}

\subsection{Proof of Theorem \ref{thm:degreecan}}

The proof is quite involved and requires a detailed analysis of the case  $q=6$, hence we break it into several  auxiliary lemmas. 
The first one is of independent interest.
\medskip
\begin{lem}\label{lem:K2q6} Let $S$ be a minimal surface of general type with $q=6$ and $p_g=2q-3=9$. If $S$ has no irrational pencil of genus $>1$, then $$K^2_S\le 35-r, $$ 
where $r$ is the number of irreducible curves contracted by the Albanese map of $S$.
\end{lem}
\begin{proof}
By Noether's formula and Hodge duality, we have $K^2_S+h^{1,1}(S)=52$. We give a lower bound for $h^{1,1}(S)$ by  using methods and results from  \cite{causin}. Let $\Ga_1,\dots \Ga_r$ be the irreducible curves contracted by  the Albanese map $a\colon S \to A$. Since the image of  $a$ is a surface, the intersection matrix  $(\Ga_i\Ga_j)_{i,j=1,\dots r}$  is negative  definite, hence the classes of the $\Ga_i$ span an $r$-dimensional  subspace $T_1\subset H^{1,1}(S)$. Since $T_1$ is orthogonal to $T_2:=a^*(H^{1,1}(A))\subset H^{1,1}(S)$,  we have  $h^{1,1}(S)\ge r+\dim T_2$. 

Denote by ${\mathbb H}_q$ the real vector space of $q\times q$ Hermitian matrices and define $d_{q,n}$ as the maximum  dimension of a subspace $V\subset {\mathbb H}_q$ such that every $0\ne M\in V$ has rank $\ge 2n$. By \cite[Proposition  2.2.3]{causin}, one has $\dim T_2\ge 30-d_{6,2}$. 

We give a rough lower bound for $d_{6,2}$ as follows. Identify ${\mathbb H}_5$ with the subspace of ${\mathbb H}_6$ consisting of the matrices whose last row and column are zero. Then if $V\subset {\mathbb H}_6$ is a subspace such that every $0\ne M\in V$ has rank at least 4, then   $\dim V\cap {\mathbb H}_5\le d_{5,4}$. We have $d_{5,4}\le 8$ by \cite[Proposition 2.2.2]{causin}, hence using Grassmann formula we get
$$\dim V \le \dim {\mathbb H}_6-\dim {\mathbb H}_5+d_{5,4}\le  36-25+8=19,$$ which gives $d_{6,2}\le 19$. Thus we get $\dim T_2\ge 17$,  $h^{1,1}(S)\ge 17+r$ and $K^2_S\le 35-r$.

\end{proof}
\medskip 
The next Lemma contains the proof of Theorem \ref{thm:degreecan} for $q\geq 7$.
\medskip

\begin{lem}\label{lem:3rational}
 Assume that $S$ is minimal  with $q\ge 6$, $p_g=2q-3$ that  $\fie$ is not birational and that $S$  has no irrational pencil of genus $\ge 2$.  Then $q=6$, $\deg\fie=3$ 
  and the canonical image $\Si\subset \pp^8$ is a rational surface of degree $11$.
\end{lem}
\begin{proof}
 Since by \cite{irregularite} a surface of general type  $S$ whose canonical system is composed with a pencil has  $q\leq 2$,  the canonical image  $\Sigma$ is a surface. 
Hence, denoting by  $d$ be the degree of $\fie$ and by  $m$ the degree of $\Sigma$, we have  $K_S^2\geq  d m$. 
 
 Assume that $d>1$.
  By \cite[Th\'eor\`eme 3.1]{beauville} either $p_g(\Sigma)=0$ or $p_g(\Sigma)=p_g(S)$ and $\Sigma$ is a canonical surface.  In the second case $m\geq 3p_g-7$ and so  $K_S^2\geq 6 p_g-14= 12q-36=9\chi(S)+3q-14$. Since $q\geq 6$,  this is a contradiction to the Miyaoka-Yau inequality.
 
 So $p_g(\Sigma)=0$ and, by Corollary  \ref{ruled}, $d\geq 3$ and $q(\Sigma)\leq 1$.  Since $d\geq 3$, we have  $ K_S^2\geq 3 m$  and therefore, since $K_S^2\leq 9(q-2)$ by the Miyaoka-Yau inequality, we get
 \begin{equation}\label{degSigma}
 m\le 3(q-2)=\frac{3}{2}(2q-4).
 \end{equation}
  Thus $\Sigma$  is a ruled surface by \cite[Lemme 1.4]{beauville}.  
 
Assume that $q(\Sigma)=1$.   By Proposition   \ref{generation}, $\Sigma$ has no pencil of rational curves of degree $<q-3$ and so   by \cite[(1.2)]{miles},   $m\geq (2(q-3)/(q-2))(2q-4)=4(q-3)$.   Since  $4(q-3)\leq 3(q-2)$ iff $q\leq 6$,  for $q=7$ we have obtained a contradiction. For $q=6$, the same argument gives $m=12$, hence $K^2_S\ge 3m=36$, contradicting  Lemma \ref{lem:K2q6}.

Hence  $q(\Sigma)=0$ and $\Sigma$ is rational.   By \eqref{degSigma}, the surface $\Sigma$ satisfies the assumptions of Theorem \ref{thm:degree}. Hence, if  $q\ge 7$ the surface $\Si$ has a pencil $|L|$ of curves such  that the span of every $L\in|L|$ has dimension $<q-2$. Since this contradicts  Proposition \ref{generation}, statement (ii) is proven.

By the same arguments, if $q=6$ and $d>1$ then $m=11$ or $m=12$. 
By Lemma \ref{lem:K2q6}, we get $35\ge K^2_S\ge dm\ge 3m$. Hence the only possibility is $d=3$ and $m=11$.
\end{proof}

\begin{lem}\label{norational} If $S$ has  no irrational pencil of genus $\ge 2$, $q=6$, $p_g=9$ and $\fie$ is not birational, then  $|K_S|$ has no fixed component and $S$ contains no rational curves.  In particular $K_S$ is ample.
\end{lem}
\begin{proof} As usual, write $|K_S|=|M|+Z$, where $|M|$ is the moving part and $Z$ is the fixed part. 
Since every global $2$-form $\si$ of $S$ can be written $\si=\al\wedge\be$ for some $\al,\be\in H^0(\Omega^1_S)$ (cf. \cite[\S 3]{mr}), the components of $Z$  are the curves on which the differential of the Albanese $a$ map drops rank.

 Let $r$ be the number of irreducible curves of $S$ contracted by $a$.
By Lemma \ref{lem:3rational} and Lemma \ref{lem:K2q6}, we have:
\begin{equation}\label{eq:K2}
35-r\ge K^2_S=K_SM+K_SZ\ge M^2+MZ\ge MZ+33.
\end{equation}
By the $2$-connectedness of canonical divisors, if $Z>0$ then   $MZ=2$, $K_SZ=0$, $Z^2=-2$. Hence every component of  $Z$ is a  smooth rational curve with self-intersection $-2$ and $r>0$, contradicting \eqref{eq:K2}. Thus  $Z=0$.  Furthermore since any rational curve   of $S$ would be contained in $Z$,  $S$ has no rational curves and so $K_S$ is ample.
\end{proof}
 \medskip

Finally we are in a position to show that also in the case $q=6$ the canonical map is birational.
\medskip
\begin{lem} Let $S$ be a minimal surface of general type with $q=6$ and $p_g=2q-3=9$. If $S$ has no irregular pencil of genus $>1$, then the canonical map of $S$ is birational.
\end{lem}
\begin{proof}
Suppose for  contradiction that $\fie$ is not birational.  Then, by Lemma \ref{lem:3rational}, $\fie$ has degree 3 and  the canonical image $\Sigma\subset \pp^8$ is a rational surface of degree $11$.

By Proposition \ref{generation} $\Sigma$ has no pencil of curves of degree $\le 3$, hence by  Proposition \ref{deg11}  we can write $K_S=2D+\Gamma$ where $\Gamma$ is an effective divisor $\geq 0$  and $|D|$ is a linear system without fixed components such that $h^0(S, D)\geq 3$.    Since,  by Lemmas \ref{lem:K2q6} and \ref{lem:3rational},   $33\leq K_S^2\leq 35$,   $K_S$  is not divisible by 2 in $\Pic(S)$. This implies that   $\Gamma\neq 0$ and $\Gamma$ is also not divisible by 2 in $\Pic(S)$.  

Since $K_S^2\leq 35$ and $K_S$ is nef, $K_SD\leq 17$.  We claim that $K_SD\geq 16$.  For contradiction suppose that $K_SD\leq 15$. By  Proposition \ref{generation},  we have $p_a(D)\geq 12$, hence $D^2\geq 7$ by  the adjunction formula. This gives a contradiction to the index theorem (Corollary \ref{Hodge1}), because $7\cdot 33= 231>(15)^2=225$. So $K_SD\geq 16$.

 Then we have the following possibilities:

\begin{itemize} 
\item $K_SD=16$, $K_S\Gamma=1$ and $K_S^2=33$;
\item $K_SD=16$, $K_S\Gamma=2$ and $K_S^2=34$;
\item $K_SD=16$, $K_S\Gamma=3$ and $K_S^2=35$;
\item $K_SD=17$, $K_S\Gamma=1$ and $K_S^2=35$;
\end{itemize} 

\medskip

We start by noticing that $\Gamma^2\leq-1$. In effect, by the index theorem (Corollary \ref{Hodge1}), $\Gamma^2\leq 0$. Since by the adjunction formula $\Gamma^2\equiv K_S\Gamma  \mod 2$, $\Gamma^2=0$ can only occur if $K_S\Gamma=2$. But this possibility is excluded by \ Corollary \ref{f1},  because $\Gamma$ is 1-connected by Corollary \ref{alwaysc}.
The same reasoning shows that any irreducible component $\theta$ of  $\Gamma$ satisfies also $\theta^2\leq -1$.  
Since, by  Lemma \ref{norational},  $K_S$ is ample, $K_S\theta>0$ for every component $\theta$ of $\Gamma$. 
Furthermore, since again by  Lemma \ref{norational}, there are no rational curves in $S$, any irreducible component  $\theta$ of $\Gamma$  such that $K_S\theta=1$ must satisfy $\theta^2=-1$, whilst an  irreducible component  $\theta$ of $\Gamma$  such that $K_S\theta=2$  must satisfy $\theta^2=-2$. Similarly if $\Gamma $ is irreducible and  $K_S\Gamma=3$ then $\Gamma^2=-3$ or  $\Gamma^2=-1$.
 Note that if $K_S\Gamma=2$ and $\Gamma$ is not irreducible, $\Gamma$ must be the sum of two distinct components because $\Gamma$  is not divisible by 2 in $\Pic(S)$.  
 
 In conclusion:

\begin{enumerate}[\rm (i)] 
\item if  $K_S\Gamma=1$ then $\Gamma$ is irreducible and $\Gamma^2=-1$;
\item if $K_S\Gamma=2$,  $\Gamma$ is reduced. \item if $K_S\Gamma=3$ and $\Gamma$ is  not reduced then $\Gamma=2\theta_1+\theta_2$ where $\theta_1, \theta_2$ are  smooth elliptic curves with self-intersection $-1$.

\end{enumerate}

In case (i)  $\Gamma^2=-1$ and $K_S=2D+\Gamma$ imply that  $\Gamma D=1$. Then   $K_SD=2D^2+\Gamma D=2D^2+1$ is odd and so   $K_SD=17$. This is a contradiction to the adjunction formula because then $D^2=8$ and $K_SD=17$.  So case (i) does not occur. 

 Case (iii) can be excluded in the same way, using the fact that $K_S=2D'+\theta_2$, where $D':=D+\theta_2$, and $K_SD'=17$.

\medskip
So we are left with the cases when  $K_S\Gamma\geq 2$ and $\Gamma$  is reduced.  Then  $K_SD=16$ and so  by the adjunction formula  $D^2$ is even. From   $K_SD=2D^2+D\Gamma$,   we conclude that $D\Gamma$ is  also even.
Then  the equality $K_S^2=4D^2+4D\Gamma+\Gamma^2$  means that  $\Gamma^2\equiv K_S^2 \mod 8$.

 On the other hand, since  every component of $\Gamma$ has geometric genus $>0$ and $\Gamma$ is reduced, also $p_a(\Gamma)>0$.  We have seen above that  $\Gamma^2<0$ and so there are only the   following  possibilities:

\begin{itemize} 
\item $K_S\Gamma=2$ and $\Gamma^2=-2$ ($K_S^2=34$);
\item $K_S\Gamma=3$ , $\Gamma^2=-1$ ($K_S^2=35$);
\item $K_S\Gamma=3$, $\Gamma^2=-3$ ($K_S^2=35$).  
\end{itemize}

This is a contradiction  because in none of these cases  $\Gamma^2\equiv K_S^2\mod 8$. 

So  $\deg\fie=3$  does not occur  and  therefore $\fie$ is birational. 
\end{proof}

The above Lemma concludes the proof of Theorem \ref{thm:degreecan}

\bigskip

\begin{minipage}{13.0cm}
\parbox[t]{6.5cm}{Margarida Mendes Lopes\\
Departamento de  Matem\'atica\\
Instituto Superior T\'ecnico\\
Universidade T{\'e}cnica de Lisboa\\
Av.~Rovisco Pais\\
1049-001 Lisboa, PORTUGAL\\
mmlopes@math.ist.utl.pt
 } \hfill
\parbox[t]{5.5cm}{Rita Pardini\\
Dipartimento di Matematica\\
Universit\`a di Pisa\\
Largo B. Pontecorvo, 5\\
56127 Pisa, Italy\\
pardini@dm.unipi.it}

\vskip1.0truecm

\parbox[t]{5.5cm}{Gian Pietro Pirola\\
Dipartimento di Matematica\\
Universit\`a di Pavia\\
Via Ferrata, 1 \\
 27100 Pavia, Italy\\
\email{gianpietro.pirola@unipv.it}}
\end{minipage}


\begin{thebibliography}{ABCD}

\bibitem[Ba]{barjaosaka}
M.A.~Barja, {\em Numerical bounds of canonical varieties}, 
Osaka J. Math. {\bf 37} (2000), 701--718.
\bibitem[BNP]{bnp} M.A.~Barja, J.C.~Naranjo and G.P.~Pirola, {\em On the topological index of irregular surfaces},  J. Algebraic Geom.  {\bf 16}  no. 3 (2007), 435--458.

\bibitem[Be1]{beauville} A. Beauville, {\em L'application canonique pour les surfaces de type g\'en\'eral}, Inv. Math. {\bf 55} (1979), 121--140.

\bibitem[Be2]{appendix} A. Beauville, {\em L'in\'egalit\'e $p_g\ge 2q-4$ pour les surfaces de type g\'en\'eral}, appendix to \cite{deb1}.



\bibitem[Bo]{bombieri} E. Bombieri,  {\em Canonical models of surfaces of general type},  Inst. Hautes \'Etudes Sci. Publ. Math. {\bf 42}  (1973), 171--219.
\bibitem[BPV]{bpv}
W. Barth, C. Peters and A. Van de Ven,
{\em Compact Complex Surfaces},
Ergebnisse der Mathematik, 3. Folge, Band 4,
Springer, Berlin, 1984.


\bibitem[CP]{causin} A.~Causin and G.~Pirola, {\em Hermitian matrices and cohomology of Kaehler varieties},
Manuscripta math.  {\bf 121} (2006),  157-168.
\bibitem[Ci2] {Ciro_curve} C.~Ciliberto, {\em Hilbert functions of finite sets of points and the genus of a curve in a projective space}, in ``Space curves'', Proceedings of the Rocca di Papa Conference, 1985, Springer Lecture Notes in Math., {\bf 1266}, (1987), 24--73.
\bibitem[CFM]{cfm} C. Ciliberto, P. Francia and M. Mendes Lopes, {\em Remarks on the bicanonical map for surfaces of general type}, Math. Z. {\bf 224} (1997), 137--166.
\bibitem[CMP]{cmp} C. Ciliberto, M. Mendes Lopes and R. Pardini, {\em Surfaces with $K^2<3\chi$ and finite fundamental group}, Math. Res. Lett. {\bf 14}, no. 6 (2007), 1081-1098. 
\bibitem[De1]{deb1} O. Debarre, {\em In\'egalit\'es num\'eriques pour les surfaces de type g\'en\'eral}, with an appendix by A. Beauville, Bull. Soc. Math. France {\bf 110} 3 (1982),  319--346.
\bibitem[De2]{deb2} O. Debarre, {Tores et vari\'et\'es ab\'eliennes}, S.M.F. Sciences 1999.
\bibitem[HP]{HaconPardini} C.D.~Hacon, R.~Pardini, {\em Surfaces with $p_g=q=3$}, Trans. Amer. Math. Soc. {\bf 354} no.~7 (2002), 2631-2638. 
\bibitem[Ha1]{harris} J. Harris, {\em A bound on the geometric genus of projective varieties}, 
Ann. Scuola Norm. Sup. Pisa Cl. Sci. (4) 8 (1981), no. 1, 35--68.

\bibitem[Ha2] {Harris_curve} J.~Harris, {\em Curves in projective space}, with the collaboration of D. Eisenbud,
 OTAN Seminar, Les Presses de l'Universit\'e de  Montreal, 1982.


\bibitem[Ko1]{konno}
K. Konno, {\em On the quadric hull of a canonical surface},  Algebraic geometry,  217--235, de Gruyter, Berlin, 2002.

\bibitem[Ko2]{konnoslope} K. Konno, {\em Non-hyperelliptic fibrations of small genus and certain irregular canonical surfaces}, Ann. Sc. Norm. Sup. Pisa ser. IV  {\bf 10} (1993), 575--595.

\bibitem[LP]{bgg} R.~Lazarsfeld and  M.Popa,  {\em BGG correspondence, cohomology of compact K\"ahler manifolds, and numerical inequalities}, 	Invent. Math. {\bf 182} no.3 (2010), 605--633. 


\bibitem[M]{adjoint}  M.~Mendes Lopes, {\em Adjoint systems on surfaces},  Boll. Un. Mat. Ital. A (7)  {\bf 10}  (1996),  no. 1, 169--179.


\bibitem[MP1]{mr}
M.~Mendes Lopes and R.~Pardini,
{\em On surfaces with $p_g=2q-3$}, Adv. in Geom.  {\bf 10} (3) (2010),  549--555.


\bibitem[MP2]{severi} M. Mendes Lopes, R. Pardini, {\em Severi type inequalities for surfaces with ample canonical class}, to appear in Comment. Math. Helv.


\bibitem[MPP]{mrp}
M.~ Mendes Lopes, R.~Pardini and Gian Pietro Pirola,
{\em On surfaces of general type with $q=5$},  arXiv:1003.5991


\bibitem[Na]{nagata}  M.~Nagata, {\em  On rational surfaces I. } Mem. Coll. Sci. Univ. Kyoto  {\bf 32} (1960), 351-370.

\bibitem[Pe]{petrakiev} I.~Petrakiev,  {\em A Step in Castelnuovo theory via Gr\"obner bases}, Journal f\"ur die reine und angewandte Mathematik, {\bf 619}   (2008),   49--73.

\bibitem[Pi]{Pirola} G.P.~Pirola, {\it Algebraic surfaces with $p_g=q=3$ and no irrational pencils}, Manuscripta Math. {\bf 108} no.~2 (2002), 163--170.

\bibitem[Ra]{ramanujam} C.P. Ramanujam,  {\em Remarks on the Kodaira vanishing theorem}, Journal of the Indian Math. Soc  {\bf 36} (1972), 41--51.

\bibitem[Re1]{miles} M. Reid,  {\em Surfaces of small degree}, Math. Ann.  {\bf 275} (1986), 71--80.

\bibitem[Re2]{Reid_quadrics} M. Reid, {\em Quadrics through a canonical surface}, in ``Algebraic Geometry -- Hyperplane
sections and related topics (L'Aquila 1988)'', Springer LNM 1417 (1990), 191--213.
\bibitem[Re3]{miles2} M. Reid, {\em  Chapters on algebraic surfaces},  in ``Complex algebraic varieties'', J. Koll\'ar (Ed.),
IAS/Park City lecture notes series (1993 volume), AMS, Providence R.I., 1997, 1--154. 

\bibitem[Sc]{schoen} C.~Schoen, {\em A family of surfaces constructed from genus 2 curves}, 
Internat. J. Math. {\bf 18} (2007), no. 5, 585--612.

\bibitem[Xi2]{irregularite} G. Xiao, {\it L' irregularit\'e des surfaces de type g\'en\'eral dont le syst\'eme canonique est compos\'e d'un pinceau},  Compositio Mathematica {\bf 56}
(1985), 251--257.
\bibitem[Xi3]{xiaohigh} G. Xiao, {\em Algebraic surfaces with high canonical degree}, Math. Ann  {\bf 274} (1986),  473--483.
\bibitem[Xi4]{xiaohyp} G. Xiao, {\em Hyperelliptic surfaces of general type with $K^2<4\chi$}, Manuscripta Math. {\bf 57} (1987), 125--148.
\bibitem[Xi5]{xiaopencil} G.~Xiao {\em Irregularity of surfaces with a linear pencil},  Duke Math. J.  {\bf 55}  (1987),  no. 3, 597--602.
\bibitem[Xi6]{xiaoslope} G. Xiao, {\it Fibered algebraic surfaces with low slope}, Math. Ann. {\bf 276}
(1987), 449--466.

\end{thebibliography}
\end{document}